\documentclass[11pt]{amsart}

\usepackage{amssymb}
\usepackage{stmaryrd}
\usepackage{mathtools}
\usepackage{booktabs}
\usepackage{subcaption}
\usepackage{microtype}

\usepackage[numbers]{natbib}
\usepackage{enumitem}

\usepackage{aliascnt}
\usepackage{hyperref}

\theoremstyle{plain}

\newaliascnt{lemma}{theorem}

\aliascntresetthe{lemma}
\newaliascnt{proposition}{theorem}
\newtheorem{proposition}[proposition]{Proposition}
\aliascntresetthe{proposition}
\newaliascnt{corollary}{theorem}

\aliascntresetthe{corollary}

\theoremstyle{definition}
\newaliascnt{definition}{theorem}
\newtheorem{definition}[definition]{Definition}
\aliascntresetthe{definition}

\theoremstyle{remark}
\newaliascnt{remark}{theorem}
\newtheorem{remark}[remark]{Remark}
\aliascntresetthe{remark}
\newaliascnt{example}{theorem}
\newtheorem{example}[example]{Example}
\aliascntresetthe{example}

\newcommand\g{\mathfrak g}

\newcommand\mr{\mathring}

\newcommand\mc{\mathcal}
\newcommand\T{\mc T}
\newcommand\abs[1]{\left\lvert{#1}\right\rvert}
\newcommand\norm[1]{\left\lVert{#1}\right\rVert}
\newcommand\wh{\widehat}
\newcommand\wt{\widetilde}
\newcommand\V{\widehat V}
\newcommand\dd[2][]{\frac{d{#1}}{d{#2}}}

\DeclareMathOperator{\grad}{grad}
\DeclareMathOperator{\curl}{curl}
\DeclareMathOperator{\Div}{div}
\DeclareMathOperator{\tr}{tr}

\DeclareMathSymbol{*}{\mathord}{symbols}{"03}

\begin{document}
\title[Charge-conserving hybrid methods for Yang--Mills]{Charge-conserving hybrid
  methods for the Yang--Mills equations}

\author{Yakov Berchenko-Kogan}
\author{Ari Stern}

\begin{abstract}
  The Yang--Mills equations generalize Maxwell's equations to
  nonabelian gauge groups, and a quantity analogous to charge is
  locally conserved by the nonlinear time evolution. \citet{ChWi2006}
  observed that, in the nonabelian case, the Galerkin method with Lie
  algebra-valued finite element differential forms appears to conserve
  charge globally but not locally, not even in a weak sense. We introduce
  a new hybridization of this method, give an alternative expression
  for the numerical charge in terms of the hybrid variables, and show
  that a local, per-element charge conservation law automatically
  holds.
\end{abstract}

\maketitle

\section{Introduction}

In 1954, \citet{YaMi1954} introduced a nonabelian gauge theory, generalizing and extending the abelian gauge theory of quantum electrodynamics. As a quantum field theory, Yang--Mills theory came to form the foundation of the Standard Model of particle physics. One may also consider classical (as opposed to quantum) solutions to the Yang--Mills equations, which can be seen as a nonlinear, nonabelian generalization of Maxwell's equations. Beyond physics, the study of classical Yang--Mills solutions has played an important role in geometry and topology \citep{DoKr1990}.

A seminal 1974 paper of \citet{Wilson1974} introduced lattice gauge theory, in which quantum Yang--Mills theory is discretized using a finite-difference-like approach. However, interest in discretization and numerical simulation of the \emph{classical} Yang--Mills equations seems to be more recent, motivated by a desire to extend insights from computational electromagnetics to develop structure-preserving methods for a more general class of nonlinear field theories. In a 2006 paper, \citet{ChWi2006} write, ``The Yang--Mills equations appear relatively ripe for numerical analysis and could therefore serve as a stepping stone toward the successful simulation of more complicated equations,'' such as Einstein's equations of general relativity.

Solutions to the Yang--Mills equations must satisfy a charge conservation law. In the special case of Maxwell's equations, this conservation law says that, in the absence of current, the charge density $ \rho = \operatorname{div} {D} $ is constant in time. The equation $ \rho = \operatorname{div} {D} $ is often viewed as a constraint, but since it is automatically preserved by the evolution of $D$, the constraint need not be ``enforced'' in any way. (A similar issue arises in Einstein's equations, whose nonlinear evolution also preserves physically important constraints.) One would also like the charge conservation law to continue to hold in numerical simulations of the Yang--Mills equations, but this is not necessarily the case, even for Maxwell's equations.

\citet{ChWi2006} observe that a standard Galerkin semidiscretization of the Yang--Mills equations only yields conservation of the total charge on the whole domain. Locally, charge is not conserved, as they illustrate in Figure 3 of their paper. Christiansen and Winther solve this problem with a constrained scheme that artificially imposes the charge conservation constraint. A different low-order charge-conserving method, based on lattice gauge theory, was given by \citet{ChHa2012}; this method preserves the constraint automatically but requires commiting a ``variational crime'' by modifying the Yang--Mills variational principle.

In contrast, we present an alternate approach, which automatically preserves a local charge conservation law without modifying the Yang--Mills variational principle. As in our work on Maxwell's equations in \citep{BeSt2020}, we consider the domain-decomposed problem, where we use discontinuous finite element spaces for our vector and scalar potentials, and then impose inter-element continuity and boundary conditions with Lagrange multipliers $\wh H$ and $\wh D$. Using the hybrid variable $\wh D$, we obtain an expression for the charge. While we are not able to get strong charge conservation when we semidiscretize, as we did for Maxwell's equations, we are able to get a local conservation law: the total charge \emph{on each element} is conserved.

The reader may naturally ask why we would be motivated to take this approach. Why not simply project the solution onto the constraint manifold, as \citet{ChWi2006} did, so that $D$ itself satisfies the constraint rather than $ \widehat{ {D} } $? The reason is that Lagrangian and Hamiltonian dynamical systems often have several conservation laws, and enforcing a single one via projection can result in \emph{worse} numerical solutions. A vivid illustration is given in \citet*[Section IV.4]{HaLuWa2006}, who present numerical simulations of the Kepler problem by the symplectic Euler method, with or without enforcing conservation of energy via projection. Perhaps surprisingly, projection makes the numerical solution much worse, destroying conservation of other quantities such as angular momentum. In fact, the symplectic Euler method automatically conserves a \emph{modified} energy \citep[Chapter IX]{HaLuWa2006}. This illustrates that automatic preservation of a modified conservation law (in the case of this paper, conservation of charge using $ \widehat{ {D} } $ rather than $D$) may be preferable to enforcing the original conservation law by projection, which risks destroying other structures that one might also wish to preserve.

The paper is structured as follows. In \autoref{sec:preliminaries}, we introduce our notation and discuss the Yang--Mills equations, leading up to the conservation of total charge in the Galerkin semidiscretization observed by Christiansen and Winther. In \autoref{sec:domaindecomposition}, we describe our domain-decomposed numerical scheme for the Yang--Mills equations and prove that it satisfies a local charge conservation property. In \autoref{sec:numerical}, we discuss our numerical implementation and illustrate with examples. Finally, in \autoref{sec:current}, we remark on how these results generalize to the Yang--Mills equations with nonzero current.

\section{Preliminaries}\label{sec:preliminaries}

\subsection{Lie algebra-valued differential forms}
In this section, we introduce Lie algebra-valued differential forms, largely following \citep{DoKr1990}.

Let $G$ be a compact Lie group with Lie algebra $\mathfrak{g}$. Let $[\cdot,\cdot]\colon\g\times\g\to\g$ denote the Lie bracket on $\g$. Such a Lie algebra always has an $\mathrm{Ad}$-invariant inner product $\langle\cdot,\cdot\rangle\colon\g\times\g\to\mathbb R$ with the property that $\langle[\xi,\eta],\omega\rangle+\langle \eta,[\xi,\omega]\rangle=0$ for all $\xi,\eta,\omega\in\g$.

Any compact Lie group can be represented as a group of unitary matrices, whose algebra consists of skew-Hermitian matrices with the commutator bracket $ [ \xi, \eta ] = \xi \eta - \eta \xi $. For simplicity of notation, we will thus view both $G$ and $\g$ as sets of matrices, in which case we can choose the inner product to simply be $\langle\xi,\eta\rangle=\tr(\xi^*\eta)$. where $\xi^*$ denotes the conjugate transpose of $\xi$.

\begin{definition}
  Let $\Omega\subset\mathbb R^n$ be a bounded Lipschitz domain. A \emph{$\g$-valued $k$-form} on $\overline\Omega$ is a section of the bundle $\left(\bigwedge^kT^*\overline\Omega\right)\otimes\g$. We will denote the space of $\g$-valued $k$-forms by $\Lambda^k(\overline\Omega,\g)$. We will denote the $L^p$ Lebesgue spaces of sections of $\left(\bigwedge^kT^*\overline\Omega\right)\otimes\g$ by $L^p\Lambda^k(\overline\Omega,\g)$.
\end{definition}

\begin{example}
  In the setting of electromagnetism, $G=U(1)$, the unit complex numbers. Then $\g=i\mathbb R$, the purely imaginary numbers. Thus, in this setting, a $\g$-valued $k$-form is simply an ordinary $k$-form times the imaginary unit $i$. The Lie bracket $[\cdot,\cdot]$ is identically zero, and the inner product is simply $\langle ia,ib\rangle=ab$.
\end{example}

The space $\Lambda^k(\overline\Omega,\g)$ is spanned by forms $\alpha\otimes\xi$, where $\alpha$ is a real-valued $k$-form and $\xi$ is an element of $\g$. With this decomposition, we can define several operations on $\g$-valued $k$-forms.
\begin{definition}
  Given $u=\alpha\otimes\xi\in\Lambda^k(\overline\Omega,\g)$ and $v=\beta\otimes\eta\in\Lambda^l(\overline\Omega,\g)$, define
  \begin{align*}
    du&=d\alpha\otimes\xi\in\Lambda^{k+1}(\overline\Omega,\g),\\
    *u&=*\alpha\otimes\xi\in\Lambda^{n-k}(\overline\Omega,\g),\\
    [u\wedge v]&=(\alpha\wedge\beta)\otimes[\xi,\eta]\in\Lambda^{k+l}(\overline\Omega,\g),\\
    \langle u\wedge v\rangle&=(\alpha\wedge\beta)\,\langle \xi,\eta\rangle\in\Lambda^{k+l}(\overline\Omega,\mathbb R),
  \end{align*}
  and extend these operations to arbitrary $\g$-valued forms by linearity.

  In the case where either $u$ or $v$ is a $0$-form, i.e., just a Lie algebra-valued function, we will often write $ [ \cdot, \cdot ] $ and $ \langle \cdot , \cdot \rangle $ instead of $ [ \cdot \wedge \cdot ] $ and $ \langle \cdot \wedge \cdot \rangle $.
\end{definition}
We have the following identities for $\mathfrak{g}$-valued forms.

\begin{proposition}\label{prop:identities}
  For $ u \in \Lambda ^k ( \overline{ \Omega } , \mathfrak{g} ) $,
  $ v \in \Lambda ^l ( \overline{ \Omega }, \mathfrak{g} )$, we have
  the Leibniz rules
  \begin{align}
    {d} [ u \wedge v ] &= [ {d} u \wedge v ] + ( -1 ) ^k [ u \wedge dv ] , \label{eqn:squareLeibniz} \\
    {d} \langle  u \wedge v \rangle  &= \langle  {d} u \wedge v \rangle  + ( -1 ) ^k \langle  u \wedge dv \rangle \label{eqn:angleLeibniz},
  \end{align}
  and the commutativity relations
  \begin{align}
    [ u \wedge v ] + ( - 1 ) ^{k l} [ v \wedge u ] &= 0, \label{eqn:squareCommute} \\
    \langle u \wedge v \rangle - ( -1 ) ^{ k l } \langle v \wedge u \rangle &= 0 \label{eqn:angleCommute}.
  \end{align}
  Additionally, given
  $ w \in \Lambda ^p ( \overline{ \Omega } , \mathfrak{g} ) $,
  \begin{align}
    \bigl[ [ u \wedge v ] \wedge w \bigr] + ( -1 ) ^{ k l } \bigl[ v \wedge [u \wedge w ] \bigr] &= \bigl[ u \wedge [ v \wedge w ] \bigr] , \label{eqn:squareAssoc} \\
    \bigl\langle [ u \wedge v ] \wedge w \bigr\rangle + ( - 1 ) ^{ k l } \bigl\langle v \wedge [ u \wedge w ] \bigr\rangle &= 0 \label{eqn:angleAssoc}.
  \end{align}
\end{proposition}

\begin{proof}
  It suffices to prove these identities for forms of the type
  $ u = \alpha \otimes \xi $, $ v = \beta \otimes \eta $,
  $ w = \gamma \otimes \omega $, since they extend to arbitrary forms
  by linearity.

  The Leibniz rules \eqref{eqn:squareLeibniz} and
  \eqref{eqn:angleLeibniz} follow immediately from the Leibniz rule
  $ {d} ( \alpha \wedge \beta ) = {d} \alpha \wedge \beta + ( - 1 ) ^k
  \alpha \wedge {d} \beta $ for ordinary real-valued forms.

  The commutativity relations \eqref{eqn:squareCommute} and
  \eqref{eqn:angleCommute} follow from
  $ \alpha \wedge \beta = ( - 1 ) ^{ k l } \beta \wedge \alpha $,
  together with the antisymmetry of $ [ \cdot , \cdot ] $ and symmetry
  of $ \langle \cdot , \cdot \rangle $, respectively.

  Finally, \eqref{eqn:squareAssoc} and \eqref{eqn:angleAssoc} follow
  from
  $ \alpha \wedge \beta \wedge \gamma = ( - 1 ) ^{ k l } \beta \wedge
  \alpha \wedge \gamma $, together with the Jacobi identity for
  $ [ \cdot , \cdot ] $ and the invariance property
  $\langle[\xi,\eta],\omega\rangle+\langle \eta,[\xi,\omega]\rangle=0$
  of $ \langle \cdot , \cdot \rangle $, respectively.
\end{proof}

In the classical formulation of electromagnetics, the electric field $E$ and electric flux density $ {D} = \epsilon E $ are vector fields, where $\epsilon$ is the \emph{electric permittivity tensor}. Likewise, the magnetic flux density $B$ and magnetic field $ H = \mu ^{-1} B $ are vector fields, where $\mu$ is the \emph{magnetic permeability tensor}. When expressed in terms of differential forms, $E$ and $H$ are $1$-forms, $D$ and $B$ are $2$-forms, and $\epsilon$ and $ \mu ^{-1} $ correspond to the Hodge star operator mapping $1$-forms and $2$-forms to $ (3-1) $-forms and $(3-2)$-forms, respectively. In vacuum, with appropriately chosen units, each of these is simply the ordinary Hodge star operator $*$. For more on the differential forms point of view for finite element methods in computational electromagnetics, see \citet{Hiptmair2002} and references therein.

This motivates the following generalized notion of electric permittivity and magnetic permeability, in arbitrary dimension $n$, for both ordinary and $\mathfrak{g}$-valued differential forms.

\begin{definition}\label{def:epsilon_mu}
The electric permittivity tensor $\epsilon$ and magnetic permeability tensor $\mu$ are pointwise symmetric isomorphisms
\begin{align*}
\epsilon_x&\colon\textstyle\bigwedge^1T^*_x\overline\Omega\to\bigwedge^{n-1}T^*_x\overline\Omega,& \mu_x^{-1}&\colon\textstyle\bigwedge^2T^*_x\overline\Omega\to\bigwedge^{n-2}T^*_x\overline\Omega.
\end{align*}
for each $x\in\overline\Omega$. The symmetry of $\epsilon$ and $\mu^{-1}$ is in the sense that
\begin{align*}
  \alpha\wedge\epsilon\beta&=\beta\wedge\epsilon\alpha&\text{ for any }\alpha,\beta\in\Lambda^1(\overline\Omega,\mathbb R),\\
  \alpha\wedge\mu^{-1}\beta&=\beta\wedge\mu^{-1}\alpha&\text{ for any }\alpha,\beta\in\Lambda^2(\overline\Omega,\mathbb R).
\end{align*}
We can extend these isomorphisms to maps
\begin{align*}
  \epsilon_x&\colon\textstyle\bigwedge^1T^*_x\overline\Omega\otimes\g\to\bigwedge^{n-1}T^*_x\overline\Omega\otimes\g,&\mu_x^{-1}&\colon\textstyle\bigwedge^2T^*_x\overline\Omega\otimes\g\to\bigwedge^{n-2}T^*_x\overline\Omega\otimes\g
\end{align*}
by ignoring the Lie algebra coefficient; that is $\epsilon_x(\alpha_x\otimes\xi_x):=\epsilon_x\alpha_x\otimes\xi_x$.
\end{definition}

As before, these operators have (anti)symmetry properties.
\begin{proposition}
  \begin{align*}
    [u\wedge\epsilon v]&=-[v\wedge\epsilon u],&\langle u\wedge\epsilon v\rangle&=\langle v\wedge\epsilon u\rangle,&u,v\in\Lambda^1(\overline\Omega,\g)\\
    [u\wedge\mu^{-1}v]&=-[v\wedge\mu^{-1}u],&\langle u\wedge\mu^{-1}v\rangle&=\langle v\wedge\mu^{-1}u\rangle,&u,v\in\Lambda^2(\overline\Omega,\g).
  \end{align*}
  In particular, $[u\wedge\epsilon u]=0$ for $u\in\Lambda^1(\overline\Omega,\g)$ and $[u\wedge\mu^{-1}u]=0$ for $u\in\Lambda^2(\overline\Omega,\g)$.
  \begin{proof}
    As before, we can prove these claims for basic tensors $u=\alpha\otimes\xi$ and $v=\beta\otimes\eta$ using the symmetry of $\epsilon$, $\mu^{-1}$, and $\langle \cdot , \cdot \rangle$ and the antisymmetry of $[\cdot , \cdot ]$. We then extend to general $u$ and $v$ by linearity.
  \end{proof}
\end{proposition}

\subsection{Connections, curvature, and the exterior covariant derivative}
We now discuss connections, again following \citep{DoKr1990}. As in \citep{ChWi2006}, we restrict our attention to the trivial bundle case. In this setting, a \emph{connection} $A$ is just a $\g$-valued one-form.
\begin{definition}
  Let $A\in\Lambda^1(\overline\Omega,\g)$. The \emph{curvature} of $A$, denoted $F_A\in\Lambda^2(\overline\Omega,\g)$, is defined by
  \begin{equation*}
    F_A=dA+\tfrac12[A\wedge A].
  \end{equation*}
  The \emph{exterior covariant derivative with respect to $A$}, denoted $d_A\colon\Lambda^k(\overline\Omega,\g)\to\Lambda^{k+1}(\overline\Omega,\g)$, is defined by
  \begin{equation*}
    d_Au=du+[A\wedge u].
  \end{equation*}
\end{definition}

\begin{example}
  In the setting of electromagnetism with $G=U(1)$, the Lie algebra has trivial commutator $[\xi,\eta]=0$. Thus, $F_A=dA$ and $d_A=d$.
\end{example}

Unlike in electromagnetism, $d_A^2\neq0$. Instead, $d_A^2=F_A$, in the following sense:
\begin{proposition}
  Let $u\in\Lambda^k(\overline\Omega,\g)$. Then
  \begin{equation*}
    d_A(d_Au)=[F_A\wedge u]\in\Lambda^{k+2}(\overline\Omega,\g).
  \end{equation*}
\end{proposition}
Additionally, we will make use of the Bianchi identity
\begin{proposition}[Bianchi identity]
  \begin{equation*}
    d_AF_A=0.
  \end{equation*}
\end{proposition}

We have a product rule for the exterior covariant derivative
\begin{proposition}
  If $A\in\Lambda^1(\overline\Omega,\g)$, $u\in\Lambda^k(\overline\Omega,\g)$ and $v\in\Lambda^l(\overline\Omega,\g)$, then
  \begin{equation*}
    d_A[u\wedge v]=[d_Au\wedge v]+(-1)^k[u\wedge d_Av].
  \end{equation*}
  \begin{proof}
    The Leibniz rule \eqref{eqn:squareLeibniz} gives $d[u\wedge v]=[du\wedge v]+(-1)^k[u\wedge dv]$, while \eqref{eqn:squareAssoc} implies $[A\wedge[u\wedge v]]=[[A\wedge u]\wedge v]+(-1)^k[u\wedge[A\wedge v]]$. Adding these together gives the claimed identity.
  \end{proof}

\end{proposition}

Finally, we can integrate by parts using the exterior covariant derivative.
\begin{proposition}
  Let $u\in\Lambda^k(\overline\Omega,\g)$ and $v\in\Lambda^{n-k-1}(\overline\Omega,\g)$. Then
  \begin{equation*}
    \begin{split}
      \int_{\partial\overline\Omega}\langle u\wedge v\rangle&=\int_{\overline\Omega}\langle du\wedge v\rangle+(-1)^k\int_{\overline\Omega}\langle u\wedge dv\rangle\\
      &=\int_{\overline\Omega}\langle d_Au\wedge v\rangle+(-1)^k\int_{\overline\Omega}\langle u\wedge d_Av\rangle.
    \end{split}
  \end{equation*}  \begin{proof}
    The first line follows from Stokes' theorem and the Leibniz rule \eqref{eqn:angleLeibniz}. The second line follows from the fact that $\langle[A\wedge u]\wedge v\rangle+(-1)^k\langle u\wedge[A\wedge v]\rangle=0$, which is a special case of \eqref{eqn:angleAssoc}.
  \end{proof}

\end{proposition}

\subsection{Electric and magnetic fields}
In order to define the Yang--Mills analogues of the scalar and vector potentials and the electric and magnetic fields, we will need some regularity assumptions. We define the following spaces
\begin{definition}
  Let
  \begin{align*}
    V^0&=\left\{\phi\in L^\infty\Lambda^0(\overline\Omega,\g) : d\phi\in L^4\Lambda^1(\overline\Omega,\g)\right\},\\
    V^1&=\left\{A\in L^4\Lambda^1(\overline\Omega,\g) : dA\in L^2\Lambda^2(\overline\Omega,\g)\right\}.
  \end{align*}
  We let $\mathring V^0$ and $\mathring V^1$ denote the subspaces of $V^0$ and $V^1$ containing those forms $\phi$ and $A$, respectively, whose tangential traces vanish on the boundary of $\overline\Omega$ in the sense of \cite{We2004}. In the smooth setting, $\mathring V^0$ contains those scalar fields that vanish on the boundary, and, in terms of vector proxies, $\mathring V^1$ contains those vector fields that are normal to the boundary.
\end{definition}

The regularity assumptions on $A$ ensure that $F_A\in L^2\Lambda^2(\overline\Omega,\g)$. The regularity assumptions on $\phi$ ensure that $d_A\phi\in V^1$ for $A\in V^1$, which will be necessary later to show charge conservation. See Equation \eqref{ym_weak_charge_conservation} and \autoref{prop:temporalimpliesgeneral}.

We can now define the Yang--Mills analogues of the scalar and vector potentials, the electric field, and the magnetic flux density. Note that we still refer to these as ``scalar'' and ``vector'' potentials, even though they are actually $\mathfrak{g}$-valued forms in this generalized setting. Here and henceforth, we employ the commonly-used ``dot'' notation for partial differentiation with respect to time, e.g., $ \dot{ A } $ means $ \partial _t A $.

\begin{definition}
  Let the \emph{scalar potential} $\phi$ be a $C^0$ curve in $V^0$ and let the \emph{vector potential} $A$ be a $C^1$ curve in $V^1$. Then define the \emph{electric field} $E$ and \emph{magnetic flux density} $B$ by
  \begin{align*}
    E &:= -(\dot A+d_A\phi), & B &:= F_A.
  \end{align*}
\end{definition}

From this, we immediately see that $E\in L^4\Lambda^1(\overline\Omega,\g)$ and $B\in L^2\Lambda^2(\overline\Omega,\g)$.

\begin{example}
  Recall that in the setting of electromagnetism with $G=U(1)$, a $\g$-valued one-form is a real-valued one-form times the imaginary unit $i$. By omitting the imaginary unit and converting the one-form to a vector field, we obtain a correspondence between the vector potential $A$ expressed as a $\g$-valued one-form and the vector potential $A$ expressed clasically as a vector field. Similarly, the scalar potential $\phi$ in this notation is a function with purely imaginary values. By omitting the imaginary unit, we obtain the usual real-valued scalar potential.

  Recall that when $G=U(1)$, we have $F_A=dA$ and $d_A=d$, so the equations for $E$ and $B$ simplify to $E=-(\dot A+d\phi)$ and $B=dA$. Converting these differential forms to vector fields, we obtain the usual equations $E=-(\dot A+\grad\phi)$ and $B=\curl A$.
\end{example}

Using the identities $d_Ad_A\phi=[F_A,\phi]$ and $d_AF_A=0$, we obtain that
\begin{align*}
  \dot B - [\phi,B]&=d_A\dot A + d_Ad_A\phi=-d_AE\\
  d_AB&=0.
\end{align*}
In the setting of electromagnetism, these equations correspond to the Maxwell equations $\dot B=-\curl E$ and $\Div B=0$.

To define the electric flux density $D$ and the magnetic field $H$, we utilize the electric permittivity tensor $\epsilon$ and magnetic permeability tensor $\mu$ of \autoref{def:epsilon_mu}. We assume that both $\epsilon$ and $\mu^{-1}$ are $L^\infty$ maps.

\begin{definition}
  Let
  \begin{align*}
    D &:= \epsilon E\in L^4\Lambda^{n-1}(\overline\Omega,\g)\\
    H &:= \mu^{-1} B\in L^2\Lambda^{n-2}(\overline\Omega,\g).
  \end{align*}
\end{definition}
From these definitions, $D$ and $H$ need only be $C^0$ curves in $L^*\Lambda^*(\overline\Omega,\g)$. We make the stronger assumption that $D$ is in fact a $C^1$ curve in $L^4\Lambda^{n-1}(\overline\Omega,\g)$.

\subsection{The Yang--Mills Lagrangian}

For this discussion, we will set the current $J$ to be zero, and we will view the charge density $\rho$ as a $C^1$ curve in $L^1\Lambda^n(\overline\Omega,\g)$. (The generalization to nonzero current is discussed in \autoref{sec:current}.)
\begin{definition}
  The \emph{Yang--Mills Lagrangian} is
  \begin{equation}\label{eq:ym_lagrangian}
    L(A,\phi,\dot A,\dot\phi):=\int_{\overline\Omega}\left(\frac12\langle E\wedge D\rangle-\frac12\langle B\wedge H\rangle - \langle\phi,\rho\rangle\right),
  \end{equation}
  where, as before, $E:=-(\dot A+d_A\phi)$, $B:=F_A$, $D:=\epsilon E$, and $H:=\mu^{-1}B$. 
\end{definition}
Note that each term in the Lagrangian is a real-valued $n$-form in at least the $L^1$ Lebesgue space, so we can indeed integrate this expression over $\overline\Omega$.

The Euler--Lagrange equations are
\begin{subequations}\label{eq:ym_el}
  \begin{alignat}{2}
    \int_{\overline\Omega}\left(\left\langle A'\wedge(\dot D-[\phi,D])\right\rangle - \left\langle d_AA'\wedge H\right\rangle\right)&=0,\qquad&\forall A'\in\mr V^1,\label{eq:ym_el_a}\\
    \int_{\overline\Omega}\left(\left\langle d_A\phi'\wedge D\right\rangle+\left\langle\phi',\rho\right\rangle\right)&=0,\qquad&\forall\phi'\in\mr V^0.\label{eq:ym_el_phi}
  \end{alignat}
\end{subequations}

These are weak expressions of the Yang--Mills equations
\begin{subequations}\label{eq:ym_strong}
\begin{align}
  \dot D-[\phi, D]&=d_AH,\\
  d_AD&=\rho.
\end{align}
\end{subequations}

\begin{example}
  In the setting of electromagnetism with $G=U(1)$, recall that $[\cdot,\cdot]=0$ and that $d_A=d$. Thus, the Yang--Mills equations in this context are
  \begin{align*}
    \dot D&=dH,&dD&=\rho,
  \end{align*}
  which are differential form expressions of Maxwell's equations,
  \begin{align*}
    \dot D&=\curl H,&\Div D&=\rho.
  \end{align*}
\end{example}

The Yang--Mills equations imply a charge conservation law.
\begin{proposition}\label{prop:rho_evolution}
  Equations \eqref{eq:ym_strong} imply that $\rho$ satisfies
  \begin{equation*}
    \dot\rho=[\phi,\rho].
  \end{equation*}
  In particular $\abs\rho$ is conserved.
  \begin{proof}
    We compute
    \begin{equation*}
      \begin{split}
        \dot\rho&=\frac d{dt}\left(d_AD\right)\\
        &=d_A\dot D+[\dot A\wedge D]\\
        &=d_Ad_AH+d_A[\phi,D]+[\dot A\wedge D]\\
        &=[F_A\wedge H]+[d_A\phi\wedge D]+[\phi,d_AD]-[E\wedge D]-[d_A\phi\wedge D]\\
        &=[B\wedge \mu^{-1}B]+[\phi,\rho]-[E\wedge \epsilon E]\\
        &=[\phi,\rho].
      \end{split}
    \end{equation*}
    Then,
    \begin{equation*}
      \dd t\abs\rho^2\,d\mathrm{vol}=\dd t\langle\rho\wedge*\rho\rangle=2\langle\dot\rho\wedge *\rho\rangle=2\langle[\phi\wedge\rho]\wedge *\rho\rangle=2\langle\phi\wedge[\rho\wedge*\rho]\rangle=0.
    \end{equation*}
  \end{proof}

\end{proposition}

\subsection{Gauge symmetry}

\begin{definition}
  A \emph{gauge transformation} is a time-dependent $G$-valued field on $\overline\Omega$. That is, a gauge transformation is a function $g\colon\overline\Omega\times\mathbb R\to G$. A gauge transformation acts on the vector and scalar potentials by the transformation
  \begin{equation*}
    g\colon(A,\phi)\mapsto\left(gAg^{-1}-(dg)g^{-1},g\phi g^{-1}+\dot gg^{-1}\right)
  \end{equation*}
\end{definition}

To explain the notation, recall that we view $G$ and $\g$ as subsets of matrices, so $g(\alpha\otimes\xi)g^{-1}$ means $\alpha\otimes g\xi g^{-1}$, where the expression $g\xi g^{-1}$ is matrix multiplication. Meanwhile, fixing a point in time and viewing $g$ as a map $\overline\Omega\to G$, we take the derivative to obtain a map $dg\colon T_x\overline\Omega\to T_gG$. Thus we can view $dg$ as a $T_gG$-valued one-form, and so $(dg)g^{-1}$ is a one-form with values in $T_eG=\g$. Similarly, fixing a point in space, we can view $g$ as map $\mathbb R\to G$. The velocity of this path $\dot g$ is a tangent vector $T_gG$, and, again, $\dot gg^{-1}$ is in $\g$.

\begin{example}
  In the setting of electromagnetism with $G=U(1)$, recall that a $\g$-valued $k$-form is simply a real-valued $k$-form times the imaginary unit $i$. Let $\xi$ be a scalar field on $\overline\Omega$. Then, setting $g=e^{-i\xi}$, we see that $g$ is a gauge transformation, and
  \begin{equation*}
    g\colon(iA,i\phi)\mapsto\left(i(A+d\xi), i(\phi-\dot\xi)\right),
  \end{equation*}
  matching the formula for gauge transformations in electromagnetism. Seeing $A$ as a vector field and $\phi$ as a scalar field, this is $ ( A, \phi ) \mapsto ( A + \operatorname{grad} \xi, \phi - \dot{ \xi } ) $, leaving $E$ and $B$ invariant.
\end{example}

One can compute the resulting action of $g$ on $E$ and $B$. Unlike in the electromagnetic situation, if $G$ is a nonabelian group, then $E$ and $B$ are not invariant under gauge transformations. Instead, $g$ acts on $E$ and $B$ by conjugating the Lie algebra values.
\begin{align*}
  g&\colon E\mapsto gEg^{-1},&g&\colon B\mapsto gBg^{-1}.
\end{align*}
However, because $\langle g\xi g^{-1},g\eta g^{-1}\rangle=\langle \xi,\eta\rangle$ for $\xi,\eta\in\g$, the expressions $\langle E\wedge D\rangle$ and $\langle B\wedge H\rangle$ in the Lagrangian are invariant under the action of gauge transformations. Thus, provided we transform $\rho\mapsto g\rho g^{-1}$, we obtain another solution to the Yang--Mills equations.

\subsection{Temporal gauge}

By applying a gauge transformation, we can set the scalar potential $\phi$ to zero. More precisely, we solve the linear differential equation
\begin{equation*}
  \dot g = -g\phi
\end{equation*}
for $g$. This gauge transformation sends $(A,\phi)$ to $(gAg^{-1}-(dg)g^{-1},0)$.

Restricting to the case $\phi=0$, called \emph{temporal gauge}, we now have
\begin{align}\label{eqn:ym_eb}
  E&=-\dot A,&B&=F_A.
\end{align}
The Lagrangian becomes
\begin{equation*}
  L(A,\dot A):=\int_{\overline\Omega}\left(\frac12\langle E\wedge D\rangle-\frac12\langle B\wedge H\rangle\right).
\end{equation*}
The corresponding Euler--Lagrange equations are
\begin{equation}\label{eqn:ym_el}
  \int_{\overline\Omega}\left(\left\langle A'\wedge\dot D\right\rangle - \left\langle d_AA'\wedge H\right\rangle\right)=0,\qquad\forall A'\in\mr V^1.
\end{equation}
This is a weak form of the equation
\begin{equation}\label{eqn:ym_strong}
  \dot D=d_AH.
\end{equation}

Setting $\rho=d_AD$, we see that $\rho$ is constant by \autoref{prop:rho_evolution} with $\phi=0$. However, when we discretize, we will find the following variational-principle-based proof of this fact more helpful. For all $\phi'\in\mr V^0$, we have that $A'=d_A\phi'\in\mr V^1$, so plugging this value of $A'$ into \eqref{eqn:ym_el}, we find
\begin{equation}\label{ym_weak_charge_conservation}
  \begin{split}
  0&=\int_{\overline\Omega}\left(\left\langle d_A\phi'\wedge\dot D\right\rangle-\left\langle d_Ad_A\phi'\wedge H\right\rangle\right)\\
  &=\int_{\overline\Omega}\left(-\left\langle\phi', d_A\dot D\right\rangle-\left\langle[B,\phi']\wedge H\right\rangle\right)\\
  &=\int_{\overline\Omega}\left(-\left\langle\phi',\dd t(d_AD)\right\rangle+\left\langle\phi',[\dot A\wedge D]\right\rangle+\left\langle\phi',[B\wedge H]\right\rangle\right)\\
  &=\int_{\overline\Omega}\left\langle\phi',-\dd t(d_AD)-[E\wedge\epsilon E]+[B\wedge\mu^{-1}B]\right\rangle,\\
  &=\int_{\overline\Omega}\left\langle\phi',-\dd t(d_AD)\right\rangle.
\end{split}
\end{equation}
Thus, $\dd t(d_AD)=0$.

In vacuum using Gaussian units, both $\epsilon$ and $\mu$ are the Hodge star $*$, and by taking the Hodge star of \eqref{eqn:ym_strong} and substituting $D=*E=-*\dot A$ and $H=*B=*F_A$, we obtain the standard formulation of the time-dependent Yang--Mills equation
\begin{equation*}
  \ddot A=-*d_A*F_A=-d_A^*F_A.
\end{equation*}

\begin{remark}
  One may ask about other choices of gauge, such as Coulomb gauge or Lorentz gauge. The issue is that, unlike in the linear setting of electromagnetism, once we have a nonlinear problem, it may not be possible to gauge transform a given connection into Coulomb gauge; that is, given $A$, there may not be a solution $g$ to the nonlinear equation $d^*\left(gAg^{-1}-(dg)g^{-1}\right)=0$. Indeed, a seminal paper of \citet{u82} shows with some difficulty that such a gauge transformation exists, provided that the energy $\norm{F_A}_{L^2(\Omega)}^2$ is sufficiently small, which allows a reduction to the linear problem via the implicit function theorem. To make use of this fact, one would need to adaptively refine the mesh so to ensure that $\norm{F_A}_{L^2(K)}^2$ is sufficiently small on each element $K$, giving a local gauge transformation $g\rvert_K$ that transforms the connection into Coulomb gauge on $K$. We believe that this adaptive mesh refinement and gauge fixing would be a fruitful direction for further investigation that would be especially useful when simulating the Yang--Mills equations for high-energy connections, but it is beyond the scope of the current paper.
\end{remark}

\subsection{Galerkin semidiscretization}
To find numerical solutions to the Yang--Mills equations, we apply Galerkin semidiscretization by restricting the trial functions $A$ and test functions $A'$ in \eqref{eqn:ym_el} to a finite dimensional subspace $V^1_h\subset \mr V^1$. That is, we seek a curve $A_h\colon t\mapsto A_h(t)\in \mr V^1_h$ such that
\begin{equation}\label{eqn:ym_galerkin_el}
  \int_{\overline\Omega}\left(\left\langle A'_h\wedge\dot D_h\right\rangle - \left\langle d_{A_h}A'_h\wedge H_h\right\rangle\right)=0,\qquad\forall A'_h\in\mr V^1_h.
\end{equation}
Here, as in \eqref{eqn:ym_eb}, we define $E_h:=-\dot A_h$, $B_h:=F_{A_h}$, and we define $D_h:=\epsilon E_h$ and $H_h=\mu^{-1}B_h$.

Unlike the corresponding situation for Maxwell's equations, \eqref{eqn:ym_galerkin_el} is a \emph{nonlinear} finite-dimensional system of ODEs, since $F_{A_h}$ contains the quadratic term $ [A _h \wedge A _h] $ and since $A_h$ appears in $d_{A_h}A_h'$.

We would like to show that $\rho_h:=d_{A_h}D_h$ is conserved, at least in some weak sense. We still have that $[\dot A_h\wedge D_h]=-[E_h\wedge\epsilon E_h]=0$. Thus, $\dot\rho_h=d_{A_h}\dot D_h+[\dot A_h\wedge D_h]=d_{A_h}\dot D_h$. However, showing that $d_{A_h}\dot D_h$ vanishes even in a weak sense cannot be done the same way as with Maxwell's equations.

As in \eqref{ym_weak_charge_conservation}, we would like to plug $A'_h=d_{A_h}\phi'_h$ into \eqref{eqn:ym_galerkin_el}, but the requirement that $A'_h$ be in $\mr V^1_h$ is difficult to satisfy because of the $[A_h,\phi_h']$ term in $d_{A_h}\phi'_h$. In general, if $\mr V^1_h$ is a space of piecewise polynomials of degree $r$, then $A_h$ will have degree $r$, so $[A_h,\phi_h']$ will generally have degree higher than $r$, and thus be an invalid choice of $A'_h$.

As noted by \citet{ChWi2006}, there is a valid choice of $\phi_h'$, namely, constant $\g$-valued functions on ${\overline\Omega}$, giving us the conservation law
\begin{equation*}
  \int_{\overline\Omega}\left\langle\phi_h',d_{A_h}\dot D_h\right\rangle=0\qquad{\text{for any constant }\phi'_h\in\g}.
\end{equation*}
In other words, the total charge $\int_{\overline\Omega}\rho_h$ on the whole domain ${\overline\Omega}$ is conserved. However, we'd like to have local charge conservation, a much stronger condition.

\section{The domain-decomposed Yang--Mills equations}\label{sec:domaindecomposition}
\subsection{Domain decomposition}
Roughly speaking, the challenge we faced above is that $\phi_h'$ had to be constant, but to get local charge conservation, we needed $\phi_h'$ to be supported on a small region. With domain decomposition, we can resolve this issue by allowing discontinuous test functions. With a discontinuous \emph{locally} constant $\phi_h'$, we can get local charge conservation.

We decompose our domain ${\overline\Omega}\subset\mathbb R^n$ using a triangulation $\T_h$ and define discontinuous function spaces with respect to this triangulation.
\begin{definition}
  Let
    \begin{align*}
    DV^0&=\left\{\phi\in L^\infty\Lambda^0({\overline\Omega},\g) : d(\phi\rvert_K)\in L^4\Lambda^1(K,\g)\text{ for all }K\in\T_h\right\},\\
    DV^1&=\left\{A\in L^4\Lambda^1({\overline\Omega},\g) : d(A\rvert_K)\in L^2\Lambda^2(K,\g)\text{ for all }K\in\T_h\right\}.
  \end{align*}
\end{definition}
That is, $DV^0$ and $DV^1$ are discontinuous versions of the spaces $V^0$ and $V^1$; the exterior derivatives are only defined after we restrict to a particular element $K$ of the triangulation.

Via Lagrange multipliers, we can characterize when a discontinuous form in $DV^0$ or $DV^1$ is actually ``continuous'' in the sense of being in $V^0$ or $V^1$ respectively, analogously to how it is done in \citep[Section III.1.2]{BrFo1991} for scalar fields. We define our spaces of Lagrange multipliers.
\begin{definition}
  Let
  \begin{align*}
    \V^{n-1}&=\left\{\wh D\in L^{4/3}\Lambda^{n-1}({\overline\Omega},\g) : d\wh D\in L^1\Lambda^n({\overline\Omega},\g)\right\},\\
    \V^{n-2}&=\left\{\wh H\in L^2\Lambda^{n-2}({\overline\Omega},\g) : d\wh H\in L^{4/3}\Lambda^{n-1}({\overline\Omega},\g)\right\}.
  \end{align*}
\end{definition}

The level of regularity in these definitions is chosen so that $\int_{\partial K}\langle\phi,\wh D\rangle$ and $\int_{\partial K}\langle A\wedge\wh H\rangle$ are well-defined for $K\in\T_h$, $\phi\in DV^0$, $A\in DV^1$, $\wh D\in\V^{n-1}$ and $\wh H\in\V^{n-2}$ via the formula
\begin{equation*}
  \int_{\partial K}\langle u\wedge\lambda\rangle=\int_K\left(\langle du\wedge\lambda\rangle + (-1)^k\langle u\wedge d\lambda\rangle\right).
\end{equation*}
Each term is in $L^1$ via H\"older's inequality. See also \cite{MiMiSh2008,We2004}.

\begin{proposition}\label{prop:ym_continuity}
  Let $\phi\in DV^0$. Then $\phi\in\mr V^0$ if and only if
  \begin{equation*}
    \sum_{K\in\T_h}\int_{\partial K}\langle \phi,\wh D\rangle=0
  \end{equation*}
  for all $\wh D\in \V^{n-1}$.

  Likewise, let $A\in DV^1$. Then $A\in\mr V^1$ if and only if
    \begin{equation*}
    \sum_{K\in\T_h}\int_{\partial K}\langle A\wedge\wh H\rangle=0
  \end{equation*}
  for all $\wh H\in\V^{n-2}$.

  \begin{proof}
    For $k=1,2$, let $u\in V^k$. Then for $\lambda\in \V^{n-k-1}$, we have
    \begin{equation}\label{eq:domain_decomposition}
      \begin{split}
        \sum_{K\in\T_h}\int_{\partial K}\langle u\wedge\lambda\rangle&=\sum_{K\in\mathcal T_h}\int_K\left(\langle du\wedge\lambda\rangle + (-1)^k\langle u\wedge d\lambda\rangle\right)\\
        &=\int_{\overline\Omega}\left(\langle du\wedge\lambda\rangle + (-1)^k\langle u\wedge d\lambda\rangle\right)\\
        &=\int_{\partial{\overline\Omega}}\langle u\wedge\lambda\rangle.
      \end{split}
    \end{equation}
    In particular, if $u\in\mr V^k$, then this expression is zero as claimed.

    Conversely, assume that $\phi\in DV^0$ and that $\sum_{K\in\T_h}\int_{\partial K}\langle \phi,\wh D\rangle=0$ for all $\wh D\in\V^{n-1}$. We can define $d\phi$ as a distribution on ${\overline\Omega}$. To show that $d\phi\in L^4\Lambda^{k+1}({\overline\Omega},\g)$, let $\wh D\in\V^{n-1}$ have vanishing trace on $\partial{\overline\Omega}$. By the definition of the distributional derivative, we have
    \begin{equation*}
      \int_{\overline\Omega}\langle d\phi\wedge\wh D\rangle=-\int_{\overline\Omega}\langle\phi, d\wh D\rangle.
    \end{equation*}
    Computing further, using the fact that $\phi\in L^\infty\Lambda^0({\overline\Omega},\g)$, $d\wh D\in L^1\Lambda^{n-1}({\overline\Omega},\g)$, and $d(\phi\vert_K)\in L^4\Lambda^k(K,\g)$, we have that
    \begin{equation*}
      \begin{split}
        -\int_{\overline\Omega}\langle \phi, d\wh D\rangle&=-\sum_{K\in\T_h}\int_K\langle\phi, d\wh D\rangle\\
        &=\sum_{K\in\T_h}\int_K\langle d\phi\wedge\wh D\rangle-\sum_{K\in\T_h}\int_{\partial K}\langle \phi,\wh D\rangle\\
        &=\sum_{K\in\T_h}\int_K\langle d\phi\wedge\wh D\rangle.
      \end{split}
    \end{equation*}
    Using H\"older's inequality, we can bound this expression by
    \begin{equation*}
      \begin{split}
        \abs{\sum_{K\in\T_h}\int_K\langle d\phi\wedge\wh D\rangle}&\le\sum_{K\in\T_h}\abs{\int_K\langle d\phi\wedge\wh D\rangle}\\
        &\le\sum_{K\in\T_h}\norm{d\phi}_{L^4\Lambda^1(K,\g)}\norm{\wh D}_{L^{4/3}\Lambda^{n-1}(K,\g)}\\
        &\le\left(\sum_{K\in\T_h}\norm{d\phi}_{L^4\Lambda^1(K,\g)}^4\right)^{1/4}\left(\sum_{K\in\T_h}\norm{\wh D}_{L^{4/3}\Lambda^{n-1}(K,\g)}^{4/3}\right)^{3/4}\\
        &=\left(\sum_{K\in\T_h}\norm{d\phi}_{L^4\Lambda^1(K,\g)}^4\right)^{1/4}\norm{\wh D}_{L^{4/3}\Lambda^{n-1}({\overline\Omega},\g)}.
      \end{split}
    \end{equation*}
    We conclude that the functional $\wh D\mapsto\int_{\overline\Omega}\langle d\phi\wedge\wh D\rangle$ is bounded on $L^{4/3}\Lambda^{n-1}({\overline\Omega},\g)$, so $d\phi\in L^4\Lambda^1({\overline\Omega},\g)$, as desired. We conclude that $\phi\in V^0$.

    Likewise, assume that $A\in DV^1$ and that $\sum_{K\in\T_h}\int_{\partial K}\langle A\wedge\wh H\rangle=0$ for all $\wh H\in\V^{n-2}$. We define $dA$ as a distribution on ${\overline\Omega}$, and in the same way that we computed for $\phi$, we can compute that for all $\wh H\in\V^{n-2}$ with vanishing trace, we have
    \begin{equation*}
      \int_{\overline\Omega}\langle dA\wedge\wh H\rangle=\sum_{K\in\T_h}\int_K\langle dA\wedge\wh H\rangle.
    \end{equation*}
    Like we did for $\phi$, we can bound this expression using the Cauchy--Schwarz inequality.
    \begin{equation*}
      \abs{\sum_{K\in\T_h}\int_K\langle dA\wedge\wh H\rangle}\le\left(\sum_{K\in T_h}\norm{dA}^2_{L^2\Lambda^2(K,\g)}\right)^{1/2}\norm{\wh H}_{L^2\Lambda^{n-2}({\overline\Omega},\g)}.
    \end{equation*}
    We conclude that the functional $\wh H\mapsto\int_{\overline\Omega}\langle dA\wedge\wh H\rangle$ is bounded on $L^2\Lambda^{n-2}({\overline\Omega},\g)$, so $dA\in L^2\Lambda^2({\overline\Omega},\g)$, as desired. We conclude that $A\in V^1$.

    We've shown that $\phi\in V^0$ and $A\in V^1$. It remains to show that their traces are zero. For $k=0,1$, considering $\lambda\in\V^{n-k-1}$, not necessarily traceless, we have by Equation \eqref{eq:domain_decomposition} and the assumption that $\sum_{K\in\T_h}\int_{\partial K}\langle u\wedge\lambda\rangle=0$ that $\int_{\partial{\overline\Omega}}\langle u\wedge\lambda\rangle=0$ for all $\lambda$, so $u$ is traceless.
  \end{proof}
\end{proposition}

\subsection{The domain-decomposed Yang--Mills equations}
We now modify the Lagrangian from \eqref{eq:ym_lagrangian} to allow $A$ and $\phi$ to come from the discontinuous function spaces, and we enforce continuity through Lagrange multipliers $\wh H\in\V^{n-2}$ and $\wh D\in\V^{n-1}$. That is, let $A$ be a $C^1$ curve in $DV^1$, and let $\phi$ be a $C^0$ curve in $DV^0$. As before, we let $E=-(\dot A+d_A\phi)$ and $B=F_A$, but in this definition we take the derivative element-wise on each $K$; in general, $A$ and $\phi$ are not weakly differentiable on $\overline\Omega$ due to jumps across element boundaries.

As before, the regularity assumptions on $\phi$ and $A$ imply that $E\in L^4\Lambda^1({\overline\Omega},\g)$ and $B\in L^2\Lambda^2({\overline\Omega},\g)$, and so this implies that $D=\epsilon E\in L^4\Lambda^{n-1}({\overline\Omega},\g)$ and $H=\mu^{-1}B\in L^2\Lambda^{n-2}({\overline\Omega},\g)$. Again, we impose the additional assumption that $\dot D\in L^4\Lambda^{n-1}({\overline\Omega},\g)$. Our Lagrangian is now
\begin{multline*}
  L(A,\phi,\wh H,\wh D,\dot A,\dot\phi,\dot{\wh H},\dot{\wh D})
  =\sum_{K\in\T_h}\left(\int_K\left(\frac12\langle E\wedge D\rangle-\frac12\langle B\wedge H\rangle - \langle\phi,\rho\rangle\right)\right.\\{}+\left.\int_{\partial K}\left(\langle A\wedge\wh H\rangle+\langle\phi,\wh D\rangle\right)\right).
\end{multline*}

The Euler--Lagrange equations are then
\begin{subequations}
  \label{eq:dd_ym}
  \begin{align}
      \int_K\left(\left\langle A'\wedge(\dot D-[\phi,D])\right\rangle - \left\langle d_AA'\wedge H\right\rangle\right)+\int_{\partial K}\langle A'\wedge\wh H\rangle&=0,\quad
      \forall A'\in{} DV^1,
    \label{eqn:dd_ym_A}\\
      \int_K\left(\left\langle d_A\phi'\wedge D\right\rangle+\left\langle\phi',\rho\right\rangle\right)-\int_{\partial K}\langle\phi',\wh D\rangle&=0,
      \quad\forall\phi'\in{}DV^0,
    \label{eqn:dd_ym_phi}\\
    \sum_{K\in\T_h}\int_{\partial K}\langle A\wedge\wh H'\rangle&=0,\quad\forall\wh H'\in \V^{n-2},\label{eqn:dd_ym_continuity_A}\\
    \sum_{K\in\T_h}\int_{\partial K}\langle\phi,\wh D'\rangle&=0,\quad\forall\wh D'\in \V^{n-1},
    \label{eqn:dd_ym_continuity_phi}
  \end{align}
\end{subequations}
where \eqref{eqn:dd_ym_A} and \eqref{eqn:dd_ym_phi} hold for all $K \in \mathcal{T} _h $. We now relate these equations to the non-domain-decomposed Euler--Lagrange equations \eqref{eq:ym_el}.

\begin{proposition}\label{prop:dd_equivalent}
  $(A,\phi,\wh H,\wh D)$ is a solution to \eqref{eq:dd_ym} if and only if $(A,\phi)$ is a solution to \eqref{eq:ym_el}, $\wh H\vert_{\partial K}=H\vert_{\partial K}$, and $\wh D\vert_{\partial K}=D\vert_{\partial K}$ for all $K$, where $\vert_{\partial K}$ denotes the tangential trace of differential forms.

  \begin{proof}
    Suppose $(A,\phi,\wh H,\wh D)$ is a solution to \eqref{eq:dd_ym}. By \autoref{prop:ym_continuity}, Equations \eqref{eqn:dd_ym_continuity_A} and \eqref{eqn:dd_ym_continuity_phi} imply that $A\in \mr V^1$ and $\phi\in\mr V^0$. Also by \autoref{prop:ym_continuity}, if we take $A'\in\mr V^1$, then $\sum_{K\in\T_h}\int_{\partial K}\langle A'\wedge\wh H\rangle=0$, so if we sum Equation \eqref{eqn:dd_ym_A} over $K$, we obtain Equation \eqref{eq:ym_el_a}. Similarly, by summing Equation \eqref{eqn:dd_ym_phi} over $K$, we obtain Equation \eqref{eq:ym_el_phi}.

    It remains to show that $\wh H\vert_{\partial K}=H\vert_{\partial K}$ and $\wh D\vert_{\partial K}=D\vert_{\partial K}$. Equations \eqref{eq:ym_el} imply that $\dot D-[\phi,D]=d_AH$ and $d_AD=\rho$ in the sense of distributions. By assumption, $\dot D\in L^4\Lambda^{n-1}({\overline\Omega},\g)$. Since $\phi\in L^\infty\Lambda^0({\overline\Omega},\g)$, we conclude then that $\dot D-[\phi,D]\in L^4\Lambda^{n-1}({\overline\Omega},\g)$, so $d_AH\in L^4\Lambda^{n-1}({\overline\Omega},\g)$. Consequently, the expression $\int_{\partial K}\langle A'\wedge H\rangle$ is well-defined by the formula
    \begin{equation}\label{eq:boundary_AH}
      \int_{\partial K}\langle A'\wedge H\rangle=\int_K\left(\langle d_AA'\wedge H\rangle-\langle A'\wedge d_AH\rangle\right).
    \end{equation}
    Indeed, the first term is the product of two $L^2$ functions, so it is in $L^1(K)$, and the second term is the product two $L^4$ functions, so it is in $L^2\subset L^1$.
    
    With this equation, and substituting $d_AH$ for $\dot D-[\phi,D]$ in \eqref{eqn:dd_ym_A}, we find that
    \begin{equation*}
      -\int_{\partial K}\langle A'\wedge H\rangle+\int_{\partial K}\langle A'\wedge\wh H\rangle=0,\qquad\forall A'\in DV^1,
    \end{equation*}
    so $\wh H\vert_{\partial K}=H\vert_{\partial K}$. Likewise, substituting $d_AD$ for $\rho$ in \eqref{eqn:dd_ym_phi} and using
    \begin{equation}\label{eq:boundary_phiD}
      \int_{\partial K}\langle\phi',D\rangle=\int_K\left(\langle d_A\phi'\wedge D\rangle+\langle \phi,d_AD\rangle\right)
    \end{equation}
    gives
    \begin{equation*}
      \int_{\partial K}\langle\phi',D\rangle-\int_{\partial K}\langle\phi',\wh D\rangle=0,\qquad\forall\phi'\in DV^0,
    \end{equation*}
    so $\wh D\vert_{\partial K}=D\vert_{\partial K}$, as desired.

    Conversely, suppose $(A,\phi)$ is a solution to \eqref{eq:ym_el}. Then $\dot D-[\phi,D]=d_AH$ and $d_AD=\rho$ in the sense of distributions. By assumption, $\dot D\in L^4\Lambda^{n-1}({\overline\Omega},\g)$. Along with $\phi\in L^\infty\Lambda^0({\overline\Omega},\g)$, $A\in L^4\Lambda^1({\overline\Omega},\g)$, and $D\in L^4\Lambda^{n-1}({\overline\Omega},\g)$, we see that
    \begin{equation*}
      dH=\dot D-[\phi,D]-[A\wedge H]\in L^{4/3}\Lambda^{n-1}({\overline\Omega},\g).
    \end{equation*}
    Indeed, the first two terms are in $L^4\subset L^{4/3}$, and the last term is in $L^4\cdot L^2=L^{4/3}$. Thus, $H\in \V^{n-2}$, and so we can set $\wh H=H$. Similarly, because $\rho\in L^1\Lambda^n({\overline\Omega},\g)$, we can use $dD=\rho-[A\wedge D]$ to conclude that $D\in\V^{n-1}$, and so we can set $\wh D=D$.

    Because $A\in\mr V^1$ and $\phi\in\mr V^0$, equations \eqref{eqn:dd_ym_continuity_A} and \eqref{eqn:dd_ym_continuity_phi} hold by \autoref{prop:ym_continuity}. By substituting $d_AH$ for $\dot D-[\phi,D]$ and $H$ for $\wh H$ and using \eqref{eq:boundary_AH}, we see that \eqref{eqn:dd_ym_A} holds. Similarly, substituting $d_AD$ for $\rho$ and $D$ for $\wh D$ and using \eqref{eq:boundary_phiD}, we see that \eqref{eqn:dd_ym_phi} holds.

  \end{proof}
\end{proposition}

\subsection{Domain decomposition in temporal gauge}

If $(A,\phi,\wh H,\wh D)$ is a solution to \eqref{eq:dd_ym}, then we can apply a gauge transformation $g$ to get a solution
\begin{equation*}
  \left(gAg^{-1}-(dg)g^{-1},g\phi g^{-1}+\dot gg^{-1},g\wh Hg^{-1},g\wh Dg^{-1}\right)
\end{equation*}
of \eqref{eq:dd_ym} with $\rho$ replaced by $g\rho g^{-1}$.

To ensure that this solution is in $DV^1\times DV^0\times \V^{n-2}\times \V^{n-1}$, it suffices to assume that $dg\in L^4\Lambda^1({\overline\Omega},\g)$ and $\dot g\in L^\infty\Lambda^0({\overline\Omega},\g)$, as we already have that $g\in L^\infty\Lambda^0({\overline\Omega},\g)$ because the group $G$ is compact.

As discussed above, we can apply a gauge transformation so that $\phi=0$ by solving $\dot g=-g\phi$. Note, however, that the situation is slightly more delicate because we need $g\in V^0$ whereas, \emph{a priori}, $\phi$ is only in $DV^0$; we must use that $\phi\in\mr V^0$ by \eqref{eqn:dd_ym_continuity_phi}.

Setting $\phi$ to zero gives us a simpler Lagrangian,
\begin{equation*}
  L(A, \wh H, \dot A, \dot{\wh H})=\sum_{K\in\T_h}\left(\int_{\overline\Omega}\left(\frac12\langle E\wedge D\rangle-\frac12\langle B\wedge H\rangle\right)+\int_{\partial K}\left(\langle A\wedge\wh H\rangle\right)\right).
\end{equation*}
The Euler--Lagrange equations then simplify to
\begin{subequations}
  \label{eq:dd_ym_temporal}
  \begin{align}
      \int_K\left(\left\langle A'\wedge\dot D\right\rangle - \left\langle d_AA'\wedge H\right\rangle\right)+\int_{\partial K}\langle A'\wedge\wh H\rangle&=0,\quad
      \forall A'\in{} DV^1,
    \label{eqn:dd_ym_A_temporal}\\
    \sum_{K\in\T_h}\int_{\partial K}\langle A\wedge\wh H'\rangle&=0,\quad\forall\wh H'\in \V^{n-2},\label{eqn:dd_ym_continuity_temporal}
  \end{align}
\end{subequations}
with $D=-\epsilon\dot A$ and $H=\mu^{-1}F_A$.

We now show that equations \eqref{eq:dd_ym_temporal} imply equations \eqref{eq:dd_ym} for an appropriate choice of $\wh D$.

\begin{proposition}\label{prop:temporalimpliesgeneral}
  Let $(A,\wh H)$ be a solution to \eqref{eq:dd_ym_temporal}. Given an initial value for $\wh D$, evolve $\wh D$ by the equation $\dot{\wh D}=d_A\wh H$. Then, assuming \eqref{eqn:dd_ym_phi} holds at the initial time, it holds for all time, so $(A,0,\wh H,\wh D)$ is a solution to \eqref{eq:dd_ym}.
  \begin{proof}
    We first note that $d_A\wh H\in\V^{n-1}$, so it makes sense to set $\dot{\wh D}$ equal to this form. Indeed, $d\wh H$ is in $L^{4/3}\Lambda^{n-1}(\overline\Omega,\g)$ by assumption, and $[A\wedge\wh H]\in L^{4/3}\Lambda^{n-1}(\overline\Omega,\g)$ because it is the product of an $L^4$ form with an $L^2$ form. Next, $dd\wh H=0$ and $d[A\wedge\wh H]=[dA\wedge\wh H]-[A\wedge d\wh H]$, and one can check that our regularity assumptions on $A$ and $\wh H$ imply that both of these terms are in $L^1\Lambda^n(\overline\Omega,\g)$.

    Note that if $\phi'\in DV^0\rvert_K$ and $A\in DV^1\rvert_K$, then $d_A\phi'\in DV^1\rvert_K$. Thus, $d_A\phi'$ is a valid choice of test function $A'$ in \eqref{eqn:dd_ym_A_temporal}, from which we obtain that
    \begin{align*}
      \int_K\left(\left\langle d_A\phi'\wedge\dot D\right\rangle -\left\langle d_Ad_A\phi'\wedge H\right\rangle\right)+\int_{\partial K}\left\langle d_A\phi'\wedge\wh H\right\rangle=0,\\
      \int_K\left(\left\langle d_A\phi'\wedge\dot D\right\rangle-\left\langle [B,\phi']\wedge H\right\rangle\right)+\int_{\partial K}\left\langle d_A\phi'\wedge\wh H\right\rangle=0,\\
      \int_K\left(\left\langle d_A\phi'\wedge\dot D\right\rangle+\left\langle\phi',[B\wedge\mu^{-1}B]\right\rangle\right)+\int_{\partial K}\left\langle d_A\phi'\wedge\wh H\right\rangle=0,\\
      \int_K\left\langle d_A\phi'\wedge\dot D\right\rangle+\int_{\partial K}\left\langle d_A\phi'\wedge\wh H\right\rangle=0.
    \end{align*}
    for all $\phi'\in DV^0$ and $K\in\T_h$.

    Recall that, in temporal gauge, $\dot\rho=0$. Thus, taking the time derivative of the left-hand side of \eqref{eqn:dd_ym_phi}, we obtain
    \begin{equation*}
      \begin{split}
        &\phantom{={}}\int_K\left(\left\langle[\dot A,\phi']\wedge D\right\rangle+\left\langle d_A\phi'\wedge\dot D\right\rangle\right)-\int_{\partial K}\left\langle\phi',\dot{\wh D}\right\rangle\\
        &=\int_K\left(\left\langle\phi',[E\wedge \epsilon E]\right\rangle+\left\langle d_A\phi'\wedge \dot D\right\rangle\right)-\int_{\partial K}\left\langle\phi',d_A\wh H\right\rangle\\
        &=\int_K\left\langle d_A\phi'\wedge \dot D\right\rangle+\int_{\partial K}\left\langle d_A\phi'\wedge\wh H\right\rangle\\
        &=0.
        \end{split}
      \end{equation*}
      Thus, if \eqref{eqn:dd_ym_phi} holds at the initial time, it holds for all time. Meanwhile, \eqref{eqn:dd_ym_A} is just \eqref{eqn:dd_ym_A_temporal} with $\phi=0$, \eqref{eqn:dd_ym_continuity_A} is the same as \eqref{eqn:dd_ym_continuity_temporal}, and \eqref{eqn:dd_ym_continuity_phi} is trivial when $\phi=0$.
  \end{proof}
\end{proposition}

\subsection{Hybrid semidiscretization}
We now discretize the Yang--Mills domain-decomposed variational problem in temporal gauge. Let $DV^0_h$, $DV^1_h$, and $\V^{n-2}_h$ be finite-dimensional subspaces of $DV^0$, $DV^1$, and $\V^{n-2}$, respectively, such that for all $K\in\T_h$, $A_h\in \left.DV^1_h\right\rvert_K$, and $\phi_h\in \left.DV^0_h\right\rvert_K$ we have
\begin{align*}
  d_{A_h}\phi_h&\in \left.DV^1_h\right\rvert_K.
\end{align*}
Recall that $d_{A_h}\phi_h=d\phi_h+[A_h,\phi_h]$. Using standard finite element spaces of differential forms, we can achieve $d\phi_h\in\left.DV^1_h\right\rvert_K$ without difficulty. However, unless $G$ is abelian and the Lie bracket is zero, we generally expect that if the coefficients of $A_h$ have polynomial degree $r$ and the coefficients of $\phi_h$ have polynomial degree $s$, then the coefficients of $[A_h,\phi_h]$ have polynomial degree $r+s$. Thus, in the nonabelian setting, we cannot expect $d_{A_h}\phi_h$ to be in the same space as $A_h$ unless $s=0$.

Consequently, we set $\left.DV^0_h\right\rvert_K$ to be the space of constant $\g$-valued $0$-forms on $K$. In other words, $DV^0_h$ is the space of piecewise constant functions ${\overline\Omega}\to\g$.

We then solve equations corresponding to \eqref{eq:dd_ym_temporal} for $A_h\in DV^1_h$ and $\wh H_h\in\V^{n-2}_h$.
\begin{subequations}
  \label{eq:dd_discrete_ym_temporal}
  \begin{align}
      \int_K\left(\left\langle A'_h\wedge\dot D_h\right\rangle - \left\langle d_{A_h}A'_h\wedge H_h\right\rangle\right)+\int_{\partial K}\langle A'_h\wedge\wh H_h\rangle&=0,\quad
      \forall A'_h\in{} DV^1_h,
    \label{eqn:dd_discrete_ym_A_temporal}\\
    \sum_{K\in\T_h}\int_{\partial K}\langle A_h\wedge\wh H'_h\rangle&=0,\quad\forall\wh H'_h\in \V^{n-2}_h,\label{eqn:dd_discrete_ym_continuity_temporal}
  \end{align}
\end{subequations}
where $D_h=-\epsilon\dot A_h$, $H_h=\mu^{-1}F_{A_h}$, and \eqref{eqn:dd_discrete_ym_A_temporal} holds for all $K\in\T_h$.

Given an initial value for $\wh D_h$, we define $\wh D_h$ for all time via the equation
\begin{equation*}
  \dot{\wh D}_h=d_{A_h}\wh H_h.
\end{equation*}
Note that if $DV^1_h$ and $\wh V^{n-2}_h$ are spaces of polynomials, then $\wh D_h$ will in general have higher polynomial degree than $\wh H_h$ because of the $[A_h\wedge\wh H_h]$ term.

We now prove the analogue of \autoref{prop:temporalimpliesgeneral}.

\begin{proposition}\label{prop:conservationlaw}
  Let $(A_h,\wh H_h)$ be a solution to \eqref{eq:dd_discrete_ym_temporal}. Given an initial value for $\wh D_h$, evolve $\wh D$ by $\dot{\wh D}_h=d_{A_h}\wh H_h$. Then, assuming
  \begin{equation}\label{eq:conservationlaw}
      \int_K\left(\left\langle d_{A_h}\phi'_h\wedge D_h\right\rangle+\langle\phi'_h,\rho\rangle\right)-\int_{\partial K}\left\langle\phi'_h,\wh D_h\right\rangle=0,\quad\forall\phi'_h\in{}DV^0_h.
  \end{equation}
  holds at the initial time, it holds for all time.
  \begin{proof}
    Let $\phi'_h\in DV^0_h$. By assumption, $d_{A_h}\phi'_h\in DV^1_h$. Thus, we can plug in $A'_h=d_{A_h}\phi'_h$ into equation \eqref{eqn:dd_discrete_ym_A_temporal}. We obtain, for all $\phi'_h\in DV^0_h$,
    \begin{equation}\label{eq:ym_pluginfora}
      \int_K\left(\left\langle d_{A_h}\phi'_h\wedge\dot D_h\right\rangle - \left\langle d_{A_h}d_{A_h}\phi'_h\wedge H_h\right\rangle\right)+\int_{\partial K}\left\langle d_{A_h}\phi'_h\wedge\wh H_h\right\rangle=0.
    \end{equation}
    The first term of \eqref{eq:ym_pluginfora} is equal to $\dd t\left\langle d_{A_h}\phi'_h\wedge D_h\right\rangle$. Indeed,
    \begin{equation*}
      \dd t\left\langle d_{A_h}\phi'_h\wedge D_h\right\rangle=\left\langle d_{A_h}\phi'_h\wedge\dot D_h\right\rangle+\left\langle[\dot A_h,\phi_h]\wedge D_h\right\rangle,
    \end{equation*}
    and
    \begin{equation*}
      \left\langle[\dot A_h,\phi'_h]\wedge D_h\right\rangle=-\left\langle\phi'_h,[\dot A_h\wedge D_h]\right\rangle=\left\langle\phi'_h,[\dot A_h\wedge\epsilon\dot A_h]\right\rangle=0,
    \end{equation*}
    by the symmetry of $\epsilon$ and the antisymmetry of the Lie bracket.
    
    The second term of \eqref{eq:ym_pluginfora} is zero. Indeed,
    \begin{equation*}
      \left\langle d_{A_h}d_{A_h}\phi'_h\wedge H_h\right\rangle=\left\langle[F_{A_h},\phi'_h]\wedge\mu^{-1}F_{A_h}\right\rangle=-\left\langle\phi'_h,[F_{A_h}\wedge\mu^{-1}F_{A_h}]\right\rangle=0.
    \end{equation*}
    Meanwhile, by integration by parts and using $\partial\partial K=0$, the third term of \eqref{eq:ym_pluginfora} is
    \begin{equation*}
      \int_{\partial K}\left\langle d_{A_h}\phi'_h\wedge\wh H_h\right\rangle=-\int_{\partial K}\left\langle\phi'_h,d_{A_h}\wh H_h\right\rangle=-\int_{\partial K}\left\langle\phi'_h,\dot{\wh D}_h\right\rangle.
    \end{equation*}
    Combining this information with the fact that $\dot\rho=0$, we have that
    \begin{equation}\label{eq:ym_cons_law}
      \dd t\left(\int_K\left(\left\langle d_{A_h}\phi'_h\wedge D_h\right\rangle+\langle\phi'_h,\rho\rangle\right)-\int_{\partial K}\left\langle\phi'_h,\wh D_h\right\rangle\right)=0,
    \end{equation}
    for all $K$ and for all $\phi'_h\in DV^0_h$, as desired.
  \end{proof}
\end{proposition}

\subsection{Local charge conservation}\label{subsec:chargeconservation}
We can interpret \autoref{prop:conservationlaw} as giving us an approximate charge $\widehat\rho_h$ that satisfies a local conservation law. Namely, for any $\phi'\in DV^0_h$, we have that $\phi'$ is constant on $K$, so $d\phi'=0$ on $K$, and so \eqref{eq:conservationlaw} simplifies to
\begin{align*}
  \int_K\left(\left\langle[A_h,\phi'_h]\wedge D_h\right\rangle+\langle\phi'_h,\rho\rangle-\langle d\phi'_h\wedge\wh D_h\rangle-\langle\phi'_h,d\wh D_h\rangle\right)&=0,\\
  \int_K\left(-\langle\phi'_h,[A_h\wedge D_h]\rangle+\langle\phi'_h,\rho\rangle-\langle\phi'_h,d\wh D_h\rangle\right)&=0,\\
  \int_K\left\langle\phi'_h,d\wh D_h+[A_h\wedge D_h]\right\rangle&=\int_K\langle\phi'_h,\rho\rangle.
\end{align*}
We know that $\dot\rho=0$. Thus, if we set
\begin{equation*}
  \widehat\rho_h:=d\wh D_h+[A_h\wedge D_h],
\end{equation*}
we have that $\widehat\rho_h$ is an approximation to the charge $\rho=d_AD=dD+[A\wedge D]$ and that
\begin{equation*}
  \dd t\int_K\langle\phi'_h,\widehat\rho_h\rangle=0,\quad\forall\phi'_h\in DV^0_h.
\end{equation*}
for all $K\in\T_h$. Note that $\widehat\rho_h$ depends on both $\wh D_h$ and $D_h$.

Since $DV^0_h$ is the space of piecewise constant $\g$-valued functions, we can state the above equation more simply as
\begin{equation*}
  \dd t\int_K\widehat\rho_h=0,\quad\forall K\in\T_h.
\end{equation*}
This equation is our local conservation law: The total charge in each element is conserved.

\section{Numerical implementation}\label{sec:numerical}
We implemented our domain decomposed hybrid method for the Yang--Mills equations in FEniCS \citep{LoMaWe2012,AlBlHaJoKeLoRiRiRoWe2015} and verified that $\widehat\rho_h$ is conserved in the sense above. As illustrated in \autoref{fig:chargeconservation}, when we simulated the Yang--Mills equations, the total charge in each element as measured by $\widehat\rho_h:=d\wh D_h+[A_h\wedge D_h]$ remained zero. In contrast, the total charge in each element as measured by $\rho_h:=d_{A_h}D_h=dD_h+[A_h\wedge D_h]$ drifted away from zero, showing the advantage of this hybrid scheme. We implemented our method on a square, a flat torus (a square with periodic boundary conditions), and the surface of a sphere. We simulated the Yang--Mills equations in vacuum, that is, with $\epsilon$ and $\mu^{-1}$ being just the Hodge star operator on the domain.

\begin{figure}
  \centering
  \includegraphics{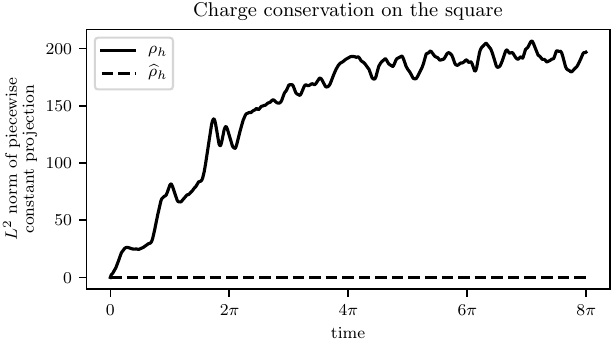}
  \vskip0.5\baselineskip
  \includegraphics{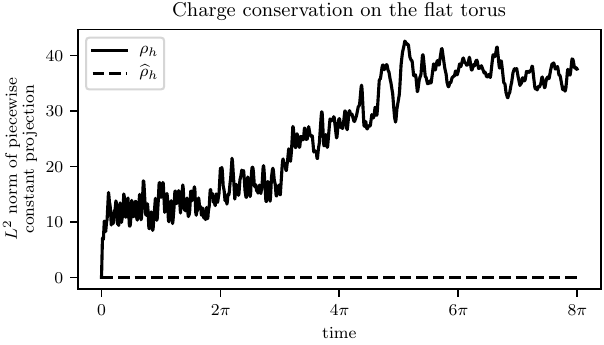}
  \vskip0.5\baselineskip
  \includegraphics{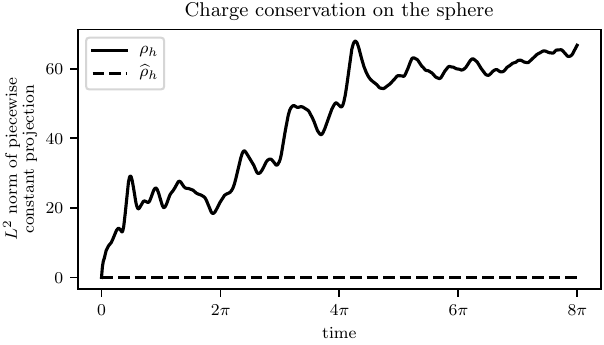}
  \caption{Numerical evolution of charge for the Yang--Mills equations
    on various domains, comparing the standard expression
    $\rho_h=d_{A_h}D_h=dD_h+[A_h\wedge D_h]$ with our new expression
    $\widehat\rho_h=d\wh D_h+[A_h\wedge D_h]$ incorporating the
    hybrid variable $ \widehat{ {D} } _h $. Projecting to
    piecewise constant $\g$-valued functions shows that the total
    charge in each element remains zero using $\widehat\rho_h$,
    whereas it drifts away from zero using $\rho_h$.}
  \label{fig:chargeconservation}
\end{figure}

We obtained solutions of the domain-decomposed problem \eqref{eq:dd_discrete_ym_temporal} in the simpler setting where our space of Lagrange multipliers $\V^{n-2}_h$ has degree large enough so that \eqref{eqn:dd_discrete_ym_continuity_temporal} forces $A_h$ to be in the conforming space $\mr V^1_h$. In this setting, we can use the evolution equation \eqref{eqn:ym_galerkin_el} from the conforming setting to evolve $A_h\in\mr V^1_h$, and then use \eqref{eqn:dd_discrete_ym_A_temporal} to solve for $\wh H_h$ as a post-processing step. We note, however, that equations \eqref{eq:dd_discrete_ym_temporal} could also be used in a more general setting where the space of Lagrange multipliers $\V^{n-2}_h$ is smaller, in which case we would obtain solutions $A_h\in DV^1_h$ that are not conforming.

We approximated $1$-forms using the $ \mathcal{P} _r \Lambda ^1 $ family of finite element differential forms \citep{ArFaWi2006,ArFaWi2010}, whose two-dimensional vector field proxies correspond to curl-conforming Brezzi--Douglas--Marini edge elements \citep{BrDoMa1985}. Tuples of these $1$-forms gives us our space $\mr V^1_h$ of $\g$-valued $1$-forms. Meanwhile, in this two-dimensional setting, $\wh H$ is a $\g$-valued zero-form, so we can represent it with a tuple of continuous Galerkin elements, giving us our space $\wh V^{n-2}_h$.

Using these curl-conforming elements, we evolved $A_h$ and $D_h$ using a leapfrog scheme, while computing the hybrid variables $\wh H_h$ and $\wh D_h$ in a post-processing step. The full numerical scheme is as follows.

\begin{enumerate}
\item Let $A_{n+\frac12}=A_n-\frac12\Delta t\epsilon^{-1}D_n$.
\item Let $\dot D_{n+\frac12}\in\mr V^1_h$ be the solution to \eqref{eqn:ym_galerkin_el}, that is,
  \begin{equation*}
    \int_{\overline\Omega}\left(\left\langle A'_h\wedge\dot D_{n+\frac12}\right\rangle - \left\langle d_{A_{n+\frac12}}A'_h\wedge H_{n+\frac12}\right\rangle\right)=0,\qquad\forall A'_h\in\mr V^1_h,
  \end{equation*}
  where $H_{n+\frac12}:=\mu^{-1}F_{A_{n+\frac12}}$.
\item \label{item:H} Let $\wh H_{n+\frac12}\in\V_h^{n-2}$ be the solution to \eqref{eqn:dd_discrete_ym_A_temporal}, that is,
  \begin{multline*}
    \int_K\left(\left\langle A'_h\wedge\dot D_{n+\frac12}\right\rangle - \left\langle d_{A_{n+\frac12}}A'_h\wedge H_{n+\frac12}\right\rangle\right)+\int_{\partial K}\langle A'_h\wedge\wh H_{n+\frac12}\rangle=0,\\
    \forall A'_h\in{} DV^1_h,\ \forall K\in\T_h,
  \end{multline*}
  that minimizes $\norm{\wh H_{n+\frac12} -H_{n+\frac12}}_{L^2(\Omega)}^2+\norm{\dot{\wh D}_{n+\frac12}-\dot D_{n+\frac12}}_{L^2(\Omega)}^2$, where $\dot{\wh D}_{n+\frac12}:=d_{A_{n+\frac12}}\wh H_{n+\frac12}$.
\item Let $D_{n+1}=D_n+\Delta t\dot D_{n+1/2}$.
\item Let $\wh D_{n+1}=\wh D_n+\Delta t\dot{\wh D}_{n+1/2}$.
\item Let $A_{n+1}=A_{n+\frac12}-\frac12\Delta t\epsilon^{-1}D_{n+1}$.
\item Let $\rho_{n+1}=d_{A_{n+1}}D_{n+1}$.
\item Let $\widehat\rho_{n+1}=d\wh D_{n+1}+[A_{n+1}\wedge D_{n+1}]$.
\end{enumerate}

The minimization in step \eqref{item:H} is needed because \eqref{eqn:dd_discrete_ym_A_temporal} does not determine $\wh H_h$ uniquely. In particular, \eqref{eqn:dd_discrete_ym_A_temporal} only involves the values of $\wh H_h$ on the element boundaries, so it gives no information about its interior degrees of freedom. Meanwhile, \eqref{eqn:dd_discrete_ym_continuity_temporal} is automatically satisfied because $A_h$ is curl-conforming.

In these examples, we worked with the three-dimensional Lie algebra $\g=\mathfrak{su}(2)$, which is isomorphic to $\mathbb R^3$ with the cross product structure, so our connection $A$ can be represented by a triple of ordinary $1$-forms, one for each component of $\g$. Let $\xi_0$, $\xi_1$, and $\xi_2$ denote a basis of $\mathfrak{su}(2)$ such that $[\xi_0,\xi_1]=\xi_2$, $[\xi_1,\xi_2]=\xi_0$, and $[\xi_2,\xi_0]=\xi_1$. For the simulations illustrated in \autoref{fig:chargeconservation}, the initial conditions we chose for $A$ are
\begin{equation*}
  \left(y(\pi-y)\,dx+x(\pi-x)\,dy\right)\otimes\xi_0+\left(y^2(\pi-y)\,dx+x^2(\pi-x)\,dy\right)\otimes\xi_1+0\otimes\xi_2
\end{equation*}
for the square,
\begin{equation*}
  \left(\sin(4x+2y)\,dx+dy\right)\otimes\xi_0+\left(dx+\sin(2x+6y\right)\,dy\otimes\xi_1+0\otimes\xi_2
\end{equation*}
for the flat torus (square with periodic boundary conditions), and the restriction of
\begin{equation*}
  \left(y(\pi-y)\,dx+x(\pi-x)\,dy+z\,dz\right)\otimes\xi_0+\left(y^2(\pi-y)\,dx+x^2(\pi-x)\,dy\right)\otimes\xi_1+0\otimes\xi_2
\end{equation*}
to the sphere for the sphere. We set $D=0$ at the initial time for all three. We chose these initial functions arbitrarily, except to ensure that they satisfy the appropriate boundary conditions and give generic-seeming solutions. In particular, the $\xi_2$ component that initially starts at zero does not remain zero, as expected since $\xi_2=[\xi_0,\xi_1]$.

\begin{table}
  \begin{subtable}{\linewidth}
    \centering
\begin{tabular*}{0.75\textwidth}{@{\extracolsep{\fill}}*{8}{r}}
\toprule
$r$ & $N$ &
\multicolumn{2}{c}{$A$} &
\multicolumn{2}{c}{$dA$} &
\multicolumn{2}{c}{$H$} \\
\midrule 1  & 4 & 10.074 & --- & 16.241 & --- & 17.637 & --- \\
 & 8 & 6.443 & 0.6 & 13.357 & 0.3 & 13.119 & 0.4 \\
 & 16 & 3.480 & 0.9 & 9.951 & 0.4 & 9.682 & 0.4 \\
 & 32 & 1.728 & 1.0 & 5.129 & 1.0 & 4.926 & 1.0 \\
 & 64 & 0.832 & 1.1 & 2.193 & 1.2 & 2.041 & 1.3 \\
 & 128 & 0.369 & 1.2 & 0.849 & 1.4 & 0.835 & 1.3 \\
 & 256 & --- & --- & --- & --- & --- & --- \\
\midrule 2  & 4 & 5.970 & --- & 13.762 & --- & 12.077 & --- \\
 & 8 & 2.004 & 1.6 & 7.245 & 0.9 & 6.656 & 0.9 \\
 & 16 & 0.554 & 1.9 & 2.331 & 1.6 & 2.203 & 1.6 \\
 & 32 & 0.144 & 1.9 & 0.603 & 2.0 & 0.568 & 2.0 \\
 & 64 & 0.036 & 2.0 & 0.223 & 1.4 & 0.216 & 1.4 \\
 & 128 & 0.009 & 2.0 & 0.081 & 1.5 & 0.078 & 1.5 \\
 & 256 & --- & --- & --- & --- & --- & --- \\
\bottomrule
\end{tabular*}
\caption{$ N \times N $ square mesh}
\end{subtable}

\bigskip

\begin{subtable}{\linewidth}
  \centering
\begin{tabular*}{0.75\textwidth}{@{\extracolsep{\fill}}*{8}{r}}
\toprule
$r$ & $N$ &
\multicolumn{2}{c}{$A$} &
\multicolumn{2}{c}{$dA$} &
\multicolumn{2}{c}{$H$} \\
\midrule 1  & 4 & 5.199 & --- & 6.887 & --- & 6.050 & --- \\
 & 8 & 2.547 & 1.0 & 6.963 & -0.0 & 7.073 & -0.2 \\
 & 16 & 1.316 & 1.0 & 4.530 & 0.6 & 4.705 & 0.6 \\
 & 32 & 0.612 & 1.1 & 2.326 & 1.0 & 2.267 & 1.1 \\
 & 64 & 0.282 & 1.1 & 0.748 & 1.6 & 0.721 & 1.7 \\
 & 128 & 0.123 & 1.2 & 0.249 & 1.6 & 0.219 & 1.7 \\
 & 256 & --- & --- & --- & --- & --- & --- \\
\midrule 2  & 4 & 2.769 & --- & 6.572 & --- & 6.577 & --- \\
 & 8 & 1.051 & 1.4 & 3.113 & 1.1 & 3.131 & 1.1 \\
 & 16 & 0.298 & 1.8 & 0.815 & 1.9 & 0.802 & 2.0 \\
 & 32 & 0.074 & 2.0 & 0.118 & 2.8 & 0.114 & 2.8 \\
 & 64 & 0.019 & 2.0 & 0.025 & 2.2 & 0.023 & 2.3 \\
 & 128 & 0.005 & 2.0 & 0.008 & 1.7 & 0.007 & 1.7 \\
 & 256 & --- & --- & --- & --- & --- & --- \\
\bottomrule
\end{tabular*}
\caption{$ N \times N $ torus mesh}
\end{subtable}  
  \caption{$L^2$ errors and rates for the numerical solution at time $\pi$, when compared to the solution on a $256 \times 256$ mesh. The results suggest linear convergence in $A$ for degree $r=1$ elements and quadratic convergence for $r=2$.}
  \label{tab:convergence}
\end{table}

\autoref{tab:convergence} shows the empirical errors and rates of convergence at $ t = \pi $ for the square and torus. By contrast with N\'ed\'elec's method for Maxwell's equations, we do not observe faster convergence of $A _h$ in the $ L ^2 $ norm than in the energy norm. In particular, $ A _h $ appears to converge with rate $r$ rather than $ r + 1 $ for degree-$r$ elements; compare the $ L ^2 $ error estimates for Maxwell's equations in Section 4 of \citet{Monk1992}. At $ t = 0 $, standard approximation theory implies that the degree-$r$ interpolant of the initial conditions has error $ \mathcal{O} ( h ^{ r + 1 } ) $, but this is seen to worsen to $ \mathcal{O} ( h ^r ) $ for longer times $t$. We suspect that the reduced rate of $ L ^2 $ convergence is due to the quadratic nonlinear term coupling the error in $A$ with the (one degree lower) error in its derivatives. Further analysis is needed but is beyond the scope of the present paper.

\begin{figure}
  \centering
  \includegraphics[width=\textwidth]{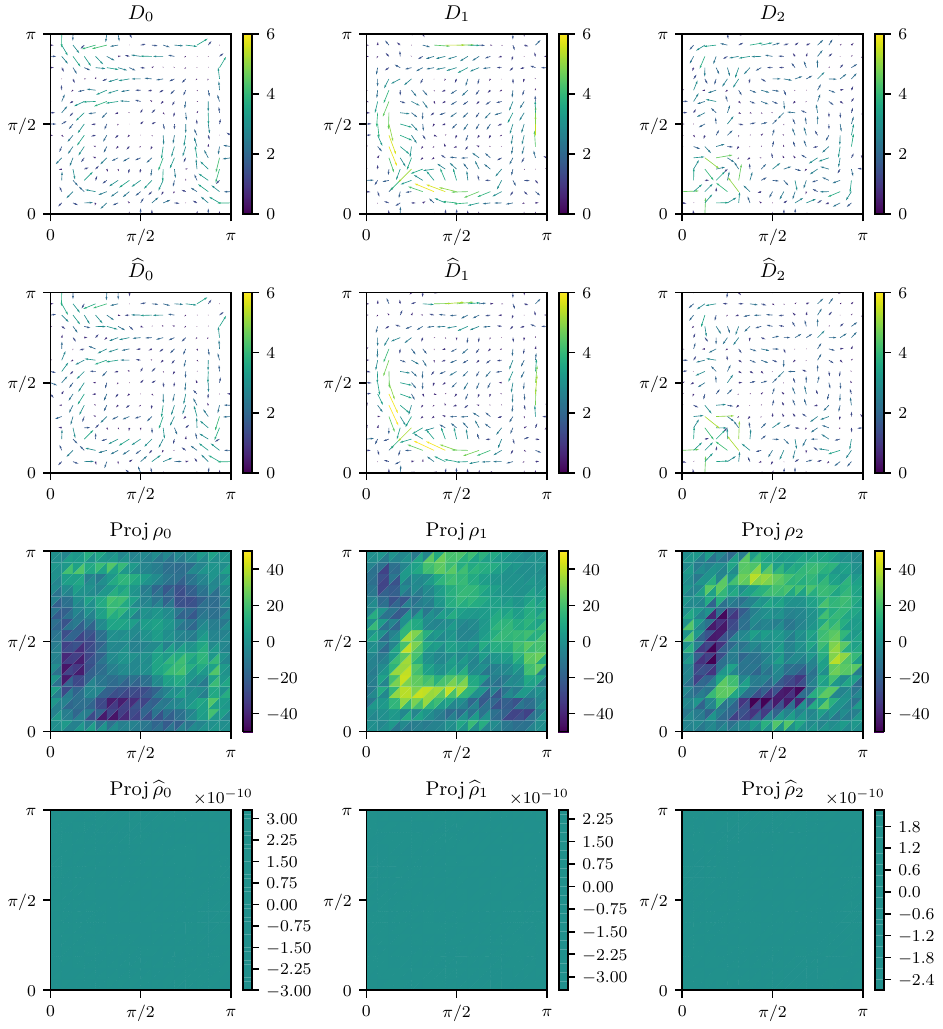}
  \caption{Comparison of the three components of $ {D} _h $ and $ \widehat{ {D} } _h $, along with the corresponding charge densities $ \rho _h $ and $ \widehat{ \rho } _h $ projected onto piecewise constants, for the $ 16 \times 16 $ square mesh at $ t = \pi $. While $ {D} _h $ and $ \widehat{ {D} } _h $ are nearly indistinguishable, $ \rho _h $ appears to show spurious nonzero charges, while $ \widehat{ \rho } _h $ remains zero due to the conservativity of the hybrid scheme.}
  \label{fig:Drho}
\end{figure}

Recall that the evolution of $\wh\rho_h$ conserves the total charge in each element $K$. To illustrate this conservation law, we projected both $\rho_h$ and $\wh\rho_h$ to the space of piecewise constant $\g$-valued functions, giving us the average charge on each element. The $L^2$ norms of these projections are plotted in \autoref{fig:chargeconservation}, showing that $\wh\rho_h$ conserved the total charge in each element, but $\rho_h$ did not. We also illustrate this behavior in \autoref{fig:Drho}, where one can see that $D _h $ and $\widehat D_h$ look identical, but there is a stark difference when we look at the corresponding charges projected to the piecewise constant functions.

\section{Remarks on the case of nonzero current}\label{sec:current}
So far, we have discussed the Yang--Mills equations with zero current, in contrast with our paper on Maxwell's equations \cite{BeSt2020}, where we do not impose this condition. For Maxwell's equations, the charge and current satisfy the continuity equation $\dot\rho=-\Div J$. We can think of $\rho$ and $J$ as given data satisfying this constraint, or, equivalently, we can think of the given data as being the initial charge distribution $\rho$ at time zero, along with the current $J$ for all time, and then we can use the equation $\dot\rho=-\Div J$ to determine the charge distribution at all future times, independently from our evolution of the potentials $\phi$ and $A$ and the corresponding fields $E$, $B$, $D$, and $H$.

In stark contrast, the corresponding relationship between $\rho$ and $J$ in the Yang--Mills setting is
\begin{equation*}
  \dot\rho-[\phi,\rho]=-d_AJ=-dJ-[A\wedge J].
\end{equation*}
As such, the evolution of the charge $\rho$ depends not only on the current $J$ but also on the potentials $\phi$ and $A$. Unlike in Maxwell's equations, we cannot determine $\rho$ \emph{a priori}; different initial conditions for $\phi$ and $A$ will lead to different future charge distributions. Of course, we have an exception to this if the current $J$ is zero, in which case, in temporal gauge, this equation reduces to $\dot\rho=0$, which does not depend on $A$.

We now discuss what happens in the general setting of nonzero current. Let $J$ be a time-varying $\g$-valued $(n-1)$-form, specifically a $C^0$ curve in $\V^{n-1}$. The Yang--Mills Lagrangian is
\begin{equation*}\label{eq:ym_lagrangian_J}
  L(A,\phi,\dot A,\dot\phi):=\int_{\overline\Omega}\left(\frac12\langle E\wedge D\rangle-\frac12\langle B\wedge H\rangle - \langle\phi,\rho\rangle+\langle A\wedge J\rangle\right).
\end{equation*}
The Euler--Lagrange equations are
  \begin{alignat*}{2}
    \int_{\overline\Omega}\left(\left\langle A'\wedge(\dot D-[\phi,D])\right\rangle - \left\langle d_AA'\wedge H\right\rangle+\langle A'\wedge J\rangle\right)&=0,\quad&\forall A'\in\mr V^1,\\
    \int_{\overline\Omega}\left(\left\langle d_A\phi'\wedge D\right\rangle+\left\langle\phi',\rho\right\rangle\right)&=0,\quad&\forall\phi'\in\mr V^0,
  \end{alignat*}
which are weak expressions of
\begin{align*}
  \dot D-[\phi, D]&=d_AH-J,\\
  d_AD&=\rho.
\end{align*}

As before, we work in temporal gauge $\phi=0$. With the standard semidiscretization, we obtain
\begin{equation*}\label{eqn:ym_galerkin_el_J}
  \int_{\overline\Omega}\left(\left\langle A'_h\wedge\dot D_h\right\rangle - \left\langle d_{A_h}A'_h\wedge H_h\right\rangle+\langle A'_h\wedge J\rangle\right)=0,\qquad\forall A'_h\in\mr V^1_h.
\end{equation*}
Meanwhile, with the domain-decomposed hybrid semidiscretization, we obtain
\begin{subequations}
  \label{eq:dd_discrete_ym_temporal_J}
  \begin{align}
      \int_K\left(\left\langle A'_h\wedge\dot D_h\right\rangle - \left\langle d_{A_h}A'_h\wedge H_h\right\rangle+\langle A_h'\wedge J\rangle\right)+\int_{\partial K}\langle A'_h\wedge\wh H_h\rangle&=0,\quad
      \forall A'_h\in{} DV^1_h,
    \label{eqn:dd_discrete_ym_A_temporal_J}\\
    \sum_{K\in\T_h}\int_{\partial K}\langle A_h\wedge\wh H'_h\rangle&=0,\quad\forall\wh H'_h\in \V^{n-2}_h,\label{eqn:dd_discrete_ym_continuity_temporal_J}
  \end{align}
\end{subequations}
and we evolve $\wh D_h$ by the equation
\begin{equation}\label{eq:whD_J}
  \dot{\wh D}_h:=d_{A_h}\wh H_h-J.
\end{equation}

So far, apart from the extra term, not much has changed from our earlier work. However, to prove the analogue of \autoref{prop:conservationlaw}, we must do something new. Previously, we had $\dot\rho=0$. Now, we have $\dot\rho=-d_AJ$, but, as discussed earlier, given $J$, we cannot determine the evolution of $\rho$ without knowing how the current interacts with $A$ via the $[A\wedge J]$ term of $d_AJ$. We only have $A_h$, not $A$, so we instead define a new quantity $\wt\rho_h$ to match $\rho$ at the initial time and evolve via
\begin{equation}\label{eq:wtrho}
  \dot{\wt\rho}_h:=-d_{A_h}J.
\end{equation}
Our results will then show that, averaged over each element, $\wh\rho_h:=d\wh D_h+[A_h\wedge D_h]$ agrees with $\wt\rho_h$. If we have reason to believe that $[A\wedge J]=[A_h\wedge J]$, then $\wt\rho_h=\rho$, and we recover our earlier results of $\wh\rho_h$ agreeing with $\rho$, but, unfortunately, we do not expect this to generally be the case. There are two special cases where $[A\wedge J]=[A_h\wedge J]$ holds. The first is when $J$ is zero, which we have addressed in the bulk of this paper. The second is when $\g$ is Abelian, in which case $A$ is simply a tuple of vector potentials independently evolving by Maxwell's equations, so we can use our stronger results in \cite{BeSt2020}.

Nonetheless, we proceed to prove the analogue of \autoref{prop:conservationlaw}.
\begin{proposition}\label{prop:conservationlaw_J}
  Let $(A_h,\wh H_h)$ be a solution to \eqref{eq:dd_discrete_ym_temporal_J}. Let $\wt\rho_h$ be defined by \eqref{eq:wtrho}, and, given an initial value for $\wh D_h$, evolve $\wh D$ by \eqref{eq:whD_J}. Then, assuming
  \begin{equation}\label{eq:conservationlaw_J}
      \int_K\left(\left\langle d_{A_h}\phi'_h\wedge D_h\right\rangle+\langle\phi'_h,\wt\rho_h\rangle\right)-\int_{\partial K}\left\langle\phi'_h,\wh D_h\right\rangle=0,\quad\forall\phi'_h\in{}DV^0_h.
  \end{equation}
  holds at the initial time, it holds for all time.
  \begin{proof}
    As in the proof of \autoref{prop:conservationlaw}, we plug in $d_{A_h}\phi'_h$ for $A'_h$ into \eqref{eqn:dd_discrete_ym_A_temporal_J}. We obtain, for all $\phi'_h\in DV^0_h$,
    \begin{equation}\label{eq:ym_pluginfora_J}
      \int_K\left(\left\langle d_{A_h}\phi'_h\wedge\dot D_h\right\rangle - \left\langle d_{A_h}d_{A_h}\phi'_h\wedge H_h\right\rangle+\left\langle d_{A_h}\phi'_h\wedge J\right\rangle\right)+\int_{\partial K}\left\langle d_{A_h}\phi'_h\wedge\wh H_h\right\rangle=0.
    \end{equation}
    Using the computations in the proof of \autoref{prop:conservationlaw}, we can reduce this equation to
    \begin{equation*}
      \int_K\left(\dd t\left\langle d_{A_h}\phi'_h\wedge D_h\right\rangle-0+\left\langle d_{A_h}\phi'_h\wedge J\right\rangle\right)-\int_{\partial K}\left\langle\phi'_h,d_{A_h}\wh H_h\right\rangle=0.
    \end{equation*}
    Dealing with the new current term, we integrate by parts to obtain
    \begin{equation*}
      \int_K\left\langle d_{A_h}\phi'_h\wedge J\right\rangle=-\int_K\left\langle\phi'_h,d_{A_h}J\right\rangle+\int_{\partial K}\left\langle\phi'_h,J\right\rangle.
    \end{equation*}
    We thus obtain
    \begin{equation*}
      \int_K\left(\dd t\left\langle d_{A_h}\phi'_h\wedge D_h\right\rangle-\left\langle\phi'_h, d_{A_h}J\right\rangle\right)-\int_{\partial K}\left\langle\phi'_h,d_{A_h}\wh H_h-J\right\rangle=0.
    \end{equation*}
    Substituting using equations \eqref{eq:wtrho} and \eqref{eq:whD_J}, we obtain
    \begin{equation*}
      \int_K\left(\dd t\left\langle d_{A_h}\phi'_h\wedge D_h\right\rangle+\left\langle\phi'_h, \dot{\wt\rho}_h\right\rangle\right)-\int_{\partial K}\left\langle\phi'_h,\dot{\wh D}_h\right\rangle=0,
    \end{equation*}
    which is the time derivative of \eqref{eq:conservationlaw_J}.
  \end{proof}
\end{proposition}

Then, as in \autoref{subsec:chargeconservation}, we can plug in piecewise constant $\phi'_h$ into \eqref{eq:conservationlaw_J} to obtain
\begin{equation*}
  \int_K\left\langle\phi'_h,d\wh D_h+[A_h\wedge D_h]\right\rangle=\int_K\langle\phi'_h,\wt\rho_h\rangle.
\end{equation*}
That is,
\begin{equation*}
  \int_K\left\langle\phi'_h,\wh\rho_h\right\rangle=\int_K\langle\phi'_h,\wt\rho_h\rangle,
\end{equation*}
or, more simply,
\begin{equation*}
  \int_K\wh\rho_h=\int_K\wt\rho_h.
\end{equation*}
In other words, with this semidiscretization, when averaged over each element, the charge as estimated by $d\wh D_h+[A_h\wedge D_h]$ automatically matches the charge as estimated by integrating $-d_{A_h}J$ with respect to time.

\subsection*{Acknowledgments}
Yakov Berchenko-Kogan was supported by an AMS--Simons Travel Grant. Ari Stern acknowledges the support of the National Science Foundation (DMS-1913272) and the Simons Foundation (\#279968). We also wish to thank the anonymous referees for their helpful comments and suggestions.

\footnotesize
\bibliographystyle{siamnat}
\bibliography{BeSt2020}

\end{document}